\documentclass[final,3p,mathptmx]{article}

\usepackage{graphicx}
\usepackage{amsmath}
\usepackage{amsfonts}
\usepackage{amssymb}
\usepackage{amsthm}
\usepackage{newlfont}
\usepackage{longtable}
\usepackage{verbatim}

\usepackage{color}
\usepackage{tikz}
\usepackage{mathrsfs}
\usepackage[titletoc,toc,title]{appendix}

\newcommand{\R}{{\mathbb R}}
\newcommand{\N}{{\mathbb N}}

\newcommand{\Rd}{{\R}^d}
\newcommand{\Comp}{\mathrm{K}(\Rd)}

\newcommand{\inK}{\in\Comp}

\newcommand{\haus}{\mathrm{haus}}
\newcommand{\dist}{\mathrm{dist}}

\newcommand{\CH}{{\mathrm{CH}}}
\newcommand{\Graph}{{\mathrm{Graph}}}

\newcommand{\PrjXonY}[2]{\Pi_{#2}{(#1)} }
\newcommand{\Pair}[2]{\Pi \big ( {#1},{#2} \big )}

\newcommand{\Var}[3]{V_{#1}^{#2}(#3) }

\newcommand{\BV}{\mathrm{BV}[a,b]}

\newcommand{\calF}{\cal{F}[a,b]}

\newcommand{\LeftLocalModul}[3]{\omega^{-}\big( {#1},{#2},{#3} \big)}

\newcommand{\NewLeftLocalModul}[3]{\varpi^{-}\big( {#1},{#2},{#3} \big)}
\newcommand{\NewRightLocalModul}[3]{\varpi^{+}\big( {#1},{#2},{#3} \big)}

\newcommand{\LocalModulCont}[3]{\omega\big( {#1},{#2},{#3} \big)}

\newcommand{\setMS}{\mathcal{S}(F)}
\newcommand{\setMSof}[1]{\mathcal{S} ({#1})}

\newcommand{\SymbText}[1]{\hbox to 10cm { \dotfill #1 \dotfill}}

\newtheorem{remark}{Remark}[section]

\newtheorem{theorem}[remark]{Theorem}
\newtheorem{propos}[remark]{Proposition}
\newtheorem{corol}[remark]{Corollary}
\newtheorem{lemma}[remark]{Lemma}
\newtheorem{result}[remark]{Result}

\newtheorem{defin}[remark]{Definition}

\begin{document}

\title {The pointwise limit of metric integral operators\\approximating set-valued functions }%

\renewcommand{\thefootnote}{\fnsymbol{footnote}}
\author{
Elena E. Berdysheva \footnotemark[1], \quad
Nira Dyn \footnotemark[2], \quad
Elza Farkhi \footnotemark[2]\;\quad
Alona Mokhov \footnotemark[4]\;
}
\footnotetext[1]{ University of Cape Town, South Africa}
\footnotetext[2]{Tel-Aviv University, School of Mathematical Sciences}
\footnotetext[4]{Afeka, Tel-Aviv Academic College of Engineering}
\date{}
\maketitle

\medskip

{ \small {\normalsize \textbf {Abstract.}  For set-valued functions  (SVFs, multifunctions), mapping a compact interval $[a,b]$ into the space of compact non-empty subsets of $\R^d$, we study approximation based on the metric approach that includes metric linear combinations, metric selections and weighted metric integrals. In our earlier papers we considered convergence of metric Fourier approximations and metric adaptations of some classical integral approximating operators for SVFs of bounded variation with compact graphs. While the pointwise limit of a sequence of these  approximants at a point of continuity $x$ of the set-valued function $F$ is $F(x)$, the limit set at a jump point was earlier described in terms of the metric selections of the multifunction. Here we show that, under certain assumptions on $F$, the limit set at $x$ equals the metric average of the left and the right limits of $F$ at $x$, thus extending the case of real-valued functions.
\medskip

\noindent{ \small {\normalsize \textbf{Key words:}} { Set-valued functions, functions of bounded variation, metric integral, metric approximation, integral operators, metric Fourier approximation, positive linear operators }

\noindent{ \small {\normalsize \textbf{Mathematics Subject Classification 2020:}} 26E25, 28B20, 41A35, 41A36, 42A20, 26A45


\section {Introduction} \label{Sect_Intro}

In a series of works, we developed a metric approach to approximation of set-valued functions of bounded variation. We consider set-valued functions (SVFs, multifunctions) that map a compact interval $[a,b]$ into the space of compact non-empty subsets of $\R^d$. Such functions find applications in different fields  such as control theory, optimization, dynamical systems, mathematical economy, and, more recently, geometric modeling. For general analysis of set-valued functions we refer to~\cite{AubinFrankowska:90}, and to~\cite{Chistyakov:On_BV-mappings} for the analysis of mappings of bounded variation.

Most of earlier results on approximation of set-valued functions study methods for multifunctions with convex values, see, for example,~\cite{Babenko:2016, BaierPerria:11, Campiti:2019, DynFarkhi:00, Lempio:95, Muresan:SVApprox2010, MNikolskii:Opt90, VIT:79}. 
Standard tools to work with set-valued functions are the Minkowski linear combinations, the support function and the Aumann integral.
Approximation methods based on these tools work well for set-valued functions with convex values, but fail to approximate functions with general, not necessarily convex values, due to the convexification phenomenon, observed first in~\cite{VIT:79} and extended in~\cite{DF:04}.

A breakthrough idea for approximating SVFs with general, not necessarily convex images, is due to Artstein~\cite{Artstein:MA}, who constructed piecewise linear approximants by connecting pairs of points that we term ``metric pairs''.
The three last authors of this paper in a series of works~\cite{DynFarkhi:01, DFM:Chains, DFM:07serdica, DFM:Book_SV-Approx, DynMokhov:06} develop the metric approach to approximation of set-valued functions based on so-called metric chains (extending the notion of metric pairs), metric linear combinations and metric selections, and apply this approach to adapt many classical sample-based approximation operators to SVFs. In~\cite{DFM:MetricIntegral} they introduce the notion of metric integral of bounded set-valued functions, which for SVFs of bounded variation can be represented by the collection of integrals of all the metric selections. The metric integral is extended to the weighted metric integral in~\cite{BDFM:2021}. The metric approach is applied by the authors in~\cite{BDFM:2019, BDFM:2021, BDFM:2022} to construct metric adaptations of a number of well-known approximation operators.

In~\cite{BDFM:2021} we prove pointwise convergence of metric trigonometric Fourier approximants.  In~\cite{BDFM:2022} we study a metric adaptation of general approximating integral operators, in particular the Bernstein-Durrmeyer and the Kantorovich operators. In~\cite{BDFM:2021, BDFM:2022} we show that sequences of metric integral operators converge pointwisely to the approximated multifunction $F$ at points of continuity of $F$. For $x$ a point of discontinuity, we show pointwise convergence to a set $A_F(x)$ described in terms of metric selections of $F$. The latter description seems to us to be quite unsatisfying, and one wishes to obtain a representation of $A_F(x)$ in terms of $F$. We achieve this goal in this paper, showing that under some assumptions on $F$, the set $A_F(x)$ is the metric average of the left and the right limits of $F$ at $x$, in full accordance with the case of real-valued functions.

The paper is organized as follows. In Section~\ref{Sect_Prelim} we recall background information and basic concepts, whereas Subsection~\ref{Sect_Prelim_OnSets} also contains some new results on the behavior of metric pairs. In Section~\ref{Sect_Erratum} we discuss relationships between the value of a set-valued function of bounded variation with a compact graph, and its one-sided limits. We also give a full proof of Proposition~7.2 from~\cite{BDFM:2021} whose proof in~\cite{BDFM:2021} was incomplete. Section~\ref{Sect_LimitSet} is devoted to the main result of the paper on the structure of the limit set $A_F(x)$. Proofs of technical statements are given in the Appendix.


\section {Preliminaries}\label{Sect_Prelim}

In this section we review some notation and basic notions related to sets and set-valued functions as well as notions of regularity of functions in metric spaces.

\subsection {On sets and set-valued functions}\label{Sect_Prelim_OnSets}

All sets considered from now on are sets in $\Rd$.  We denote by $\Comp$\label{CompSets} the collection of all compact non-empty subsets of~$\Rd$. 
The metric in $\Rd$ is of the form $\rho(u,v)=|u-v|$, where $|\cdot|$ is any fixed norm on $\Rd$. Recall that~$\Rd$ endowed with this metric is a complete metric space and that all norms on $\Rd$ are equivalent.

\noindent To measure the distance between two non-empty sets ${A,B \in \Comp}$, we use the Hausdorff metric  based on~$\rho$
$$
	\haus(A,B)_{\rho}= \max \left\{ \sup_{a \in A}\dist(a,B)_{\rho},\; \sup_{b \in B}\dist(b,A)_{\rho} \right\},
$$
where the distance from a point $c$ to a set $D$ is $\dist(c,D)_{\rho}=\inf_{d \in D}\rho(c,d)$.  
It is well known that $\Comp$ endowed with the Hausdorff metric is a complete metric space~\cite{RockWets, {S:93}}.
In the following, we keep the metric in $\Rd$ fixed, and omit the notation $\rho$ as a subscript.

 We denote by $|A|=\haus (A,\{0\})$ the ``norm'' of the set $A \in \Comp$.
The set of projections of $a \in \Rd$ on a set $B \inK$ is
$
\PrjXonY{a}{B}=\{b \in B \ : \ |a-b|=\dist(a,B)\},
$
and the set of metric pairs of two sets $A,B \inK$ is
$$
\Pair{A}{B} = \{(a,b) \in A \times B \ : \ a \in \PrjXonY{b}{A}\;\, \mbox{or}\;\, b\in\PrjXonY{a}{B} \}.
$$
Using metric pairs, we can rewrite 
\begin{equation} \label{haus_MetrPair}
\haus(A,B)= \max \{|a-b| \ :\ (a,b)\in \Pair{A}{B}\}.
\end{equation}

Now we introduce different notions of limits of sequences of sets. We say that a sequence of sets $\{A_n\}_{n=1}^{\infty}$ converges to a set $A$ in the Hausdorff metric if $\lim_{n \to \infty} \haus(A_n,A) = 0$.

	The upper Kuratowski limit of a sequence of sets $\{A_n\}_{n=1}^{\infty}$  is the set of all limit points of converging subsequences
	$\{a_{n_k}\}_{k=1}^{\infty}$, where ${a_{n_k} \in A_{n_k} }$, $k\in \N$, namely
	$$
	\limsup_{n \to \infty} A_n = \left\{a \ : \  \exists\, \{n_k\}_{k=1}^{\infty},\, n_{k+1}>n_k,\, k\in \N,\ \exists \, a_{n_k} \in A_{n_k} \text{ such that } \lim_{k \to \infty}a_{n_k} = a \right\}.
	$$
	
	The lower Kuratowski limit of $\{A_n\}_{n=1}^{\infty}$ is the set of all limit points of converging sequences
	$\{a_{n}\}_{n=1}^{\infty}$, where $a_{n} \in A_{n} $,  namely,
	$$
	\liminf_{n \to \infty} A_n = \left\{a \ : \  \exists \, a_{n} \in A_{n} \text{ such that } \lim_{n \to \infty}a_{n} = a \right\}.
	$$
For a set-valued function $F:[a,b]\to \Comp$  the upper Kuratowski limit at $\widetilde{x} \in [a,b]$ is 
$$
\limsup_{x \to \widetilde{x}} F(x) = \left\{y \ : \  \exists \, \{x_k\}_{k=1}^{\infty} \subset [a,b] \ \text{with} \  x_k\to\widetilde{x}\ , \ \exists \, \{y_k\}_{k=1}^{\infty} \ \text{with} \ y_k\in F(x_k),  k\in\N, \ \text{and} \ y_k \to y
\right\}.
$$
The lower Kuratowski limit of $F$ at $\widetilde{x} \in [a,b]$  is
$$
\liminf_{x \to \widetilde{x}} F(x) = \left\{y \ : \  \forall \, \{x_k\}_{k=1}^{\infty} \subset [a,b] \ \text{with} \  x_k\to\widetilde{x}\ , \ \exists \, \{y_k\}_{k=1}^{\infty} \ \text{with} \ y_k\in F(x_k),  k\in\N, \ \text{and} \ y_k \to y
\right\}.
$$
A set $A$ is a Kuratowski limit of $F(x)$ as $x \to \widetilde{x}$  if 
$$
A =  \liminf_{x \to \widetilde{x}} F(x)  = \limsup_{x \to \widetilde{x}} F(x).
$$
The same relations hold also for sequences of sets.
It is known that  convergence in the Hausdorff metric and in the sense of Kuratowski are equivalent, if the underlying metric space is compact  (see, e.g.,~\cite[Section 4.4]{AmbrosioTilli:Topics}).  In the following the notion of a limit is understood in the sense of Hausdorff/Kuratowski. 

Next, we discuss some properties of metric pairs.

\begin{lemma}\label{Lemma_Help} 
Let $A, \hat{A}, B \in \Comp$, $\hat{A} \subset A$. If $(a,b) \in \Pair{A}{B}$ and $a\in \hat{A} \subset A$, then $(a,b) \in \Pair{\hat{A}}{B}$.
	
\end{lemma}

The proof is straightforward.

\begin{lemma}\label{Lemma_Limit_of_MP_is_MP} 
	Let $\displaystyle \lim_{n \to \infty} A_n=A$, $\displaystyle\lim_{n \to \infty} B_n=B$, $A, B, A_n, B_n \in \Comp$, 
	$(a_n,b_n) \in \Pair{A_n}{B_n}$, $n \in \N$, and let $\lim_{n \to \infty} a_n=a$, \; $\lim_{n \to \infty} b_n=b$. 
	Then $(a,b) \in \Pair{A}{B}.$
	
\end{lemma}

\begin{proof}
	Suppose, without loss of generality, that $b_n \in \PrjXonY{a_n}{B_n}$ for infinitely many $n\in \N$. We prove that $b \in \PrjXonY{a}{B}$ by contradiction. Assume that $|b-a|> \dist (a,B)$, and take $0 < \varepsilon < |b - a| - \dist(a,B)$. There exists $b^* \in B$  such that
	${|a-b^*|=\dist(a,B)<|a-b|-\varepsilon}$. Since $\lim_{n \to \infty} B_n=B$, for $n$ large enough there exists  ${b_n^* \in \PrjXonY{b^*}{B_n}}$ such that ${|b^*-b_n^*|<\varepsilon/5}$. Moreover,  for $n$ large enough we have $|a_n - a| < \varepsilon/5$,  $|b_n - b| < \varepsilon/5$. For such $n$
	we have
	\begin{align*}\
		|a_n-b_n^*|&\le |a_n-a|+|a-b^*|+|b^*-b_n^*| < \frac{\varepsilon}{5}+|a-b|-\varepsilon+\frac{\varepsilon}{5} \\&
		\le \frac{-3\varepsilon}{5}+|a-a_n|+|a_n-b_n|+|b_n-b| 
		< \frac{-3\varepsilon}{5}+ \frac{\varepsilon}{5} + |a_n-b_n| + \frac{\varepsilon}{5}
		= | a_n-b_n|-\frac{\varepsilon}{5}.
	\end{align*}
	We obtain that $b_n^*$ is closer to $a_n$ than $b_n$,  which is in contradiction to the fact that $b_n \in \PrjXonY{a_n}{B_n}$.
\end{proof}

\begin{corol}\label{Corol_Inclusion1}
Let $\displaystyle \lim_{n \to \infty} A_n=A$, $\displaystyle \lim_{n \to \infty} B_n=B$, $A, B, A_n, B_n \in \Comp$, 
$n \in \N$. Then 
$$
\limsup_{n \to \infty} \Pair{A_n}{B_n} \subseteq \Pair{A}{B}.
$$
\end{corol}
\begin{proof}
Let $(a,b) \in \limsup\limits_{n \to \infty} \Pair{A_n}{B_n}$. By the definition of $\limsup$ there is a strictly increasing sequence $\{n_k\}_{k = 1}^\infty \subseteq \N$ such that $(a_{n_k},b_{n_k}) \in \Pair{A_{n_k}}{B_{n_k}}$ and $\lim\limits_{k \to \infty}(a_{n_k},b_{n_k})=(a,b)$.
From Lemma~\ref{Lemma_Limit_of_MP_is_MP} it follows that ${(a,b) \in \Pair{A}{B}}$.
\end{proof}


We recall the notions of  a metric chain and of a metric linear combination~\cite{DFM:MetricIntegral}.
\begin{defin}\label{Def_MetChain_MetCombination} \cite{DFM:MetricIntegral}
	Given a finite sequence of sets $A_0, \ldots, A_n \in \Comp$, $n \ge 1$,  a metric chain of $A_0, \ldots, A_n$ is an $(n+1)$-tuple $(a_0,\ldots,a_n)$ such that $(a_i,a_{i+1}) \in \Pair {A_i}{A_{i+1}}$, $i=0,1,\ldots,n-1$. The collection of all metric chains of $A_0, \ldots, A_n$ is denoted by  
	$$
	\CH(A_0,\ldots,A_n)= \left\{ (a_0,\ldots,a_n) \ : \ (a_i,a_{i+1}) \in \Pair {A_i}{A_{i+1}}, \  i=0,1,\ldots,n-1 \right\}.
	$$
	The metric linear combination of the sets $A_0, \ldots, A_n \in \Comp$, $n \ge 1$, is
	$$
	\bigoplus_{i=0}^n \lambda_i A_i =
	\left\{ \sum_{i=0}^n \lambda_i a_i \ : \ (a_0,\ldots,a_n) \in \CH(A_0,\ldots,A_n) \right\}, \quad \lambda_0,\ldots,\lambda_n \in \R.
	$$
	In case $n=1$ we write $\lambda_0 A_0 \oplus \lambda_1 A_1$. 
\end{defin}

\begin{remark}\label{Remar_ThereIsChain}
	For any $j\in \N$,  $0 \le j\le n$ and for any $a \in A_j$ there exists a metric chain $(a_0, \ldots, a_n) \in \CH(A_0,\ldots,A_n)$ such that $a_j=a$. For a possible construction see~\cite{DFM:Chains}, Figure~3.2.
\end{remark}

Note that the metric linear combination depends on the order of the sets, in contrast to the Minkowski linear combination of sets which is defined by
$$
\sum_{i=0}^n \lambda_i A_i =
\left\{ \sum_{i=0}^n \lambda_i a_i \ : \  a_i \in A_i \right\},\quad n \ge 1.
$$

\subsection {Notions of regularity of functions with values in a metric space}\label{Sect_Prelim_Regularuty}
In this section we discuss functions defined on a fixed compact interval $[a,b] \subset \R$ with values in a complete metric space $(X,\rho)$, where $X$ is either $\Rd$ or $\Comp$. 

We recall the notion of the variation of ${f:[a,b]\rightarrow X}$. Let ${\, \chi=\{x_0,\ldots, x_n\} }$, $a=x_0 < x_1 < \cdots <x_n=b$, be a partition of the interval $[a,b]$ with the norm 
$$
{\displaystyle |\chi|=\max_{0\le i\le n-1} (x_{i+1}-x_i)}.
$$ 
The variation of $f$ on the partition $\chi$ is defined as 
$
V(f,\chi) = \sum_{i=1}^{n} \rho(f(x_i),f(x_{i-1}))\, .
$
The total variation of $f$ on $[a,b]$ is 
$$
V_{a}^{b}(f) = \sup_{\chi} V(f,\chi),
$$
where the supremum is taken over all partitions of $[a,b]$.

A function $f$ is said to be of bounded variation on $[a,b]$ if ${ V_{a}^{b}(f) < \infty}$. We call functions of bounded variation BV functions and write $f \in \BV$.  

For $f \in \BV$ the  function $v_f:[a,b]\rightarrow \R$,\, $v_f(x)=V_{a}^{x}(f)$ is called the variation function of $f$. Note that 
$$
V_{z}^{x}(f)=v_f(x)-v_f(z) \quad  \mbox{for} \quad a\le z<x \le b,
$$
and that $v_f$ is monotone non-decreasing.

We recall the notion of the local modulus of continuity~\cite{SendovPopov}, which is central to the approximation of functions at continuity points.
For $f : [a,b] \to X$  the local modulus of continuity at $x^* \in [a,b]$ is
$$
	\LocalModulCont{f}{x^*}{\delta} = \sup \left\{\, \rho(f(x_1),f(x_2)): \; x_1,x_2 \in \left[x^*-\delta/2,x^*+\delta/2 \right]\cap[a,b] \,\right\},\quad \delta >0.
$$
The left and the right local moduli of continuity of $f$ at $x^*\in[a,b]$  are defined respectively  by
$$
\LeftLocalModul{f}{x^*}{\delta}=\sup \left\{ \rho(f(x),f(x^*)) \ :\ x\in [x^*-\delta, x^*] \cap [a,b] \right \},\quad \delta>0,
$$
and
$$
\omega^{+}(f,x^*,\delta)=\sup \left \{\rho(f(x),f(x^*)) \ : \ x\in [x^*,x^*+\delta] \cap [a,b] \right \},\quad \delta>0.
$$

It follows from the definition of the variation that~

\begin{result}\label{Result_f_OneSidedCont->v_f_OneSidedCont}
		A function $f:[a,b]\to X$, $f \in \BV$ is left continuous at $x^*\in (a,b]$ if and only if $v_f$ is left continuous at~$x^*$.
		The function $f$ is right continuous at $x^* \in [a,b)$ if and only if $v_f$ is right continuous at~$x^*$.
\end{result}

A function $f : [a,b] \to X$ of bounded variation with values in a complete metric space $(X,\rho)$ is not necessarily continuous, but has right and left limits at any point $x$~\cite{Chistyakov:On_BV-mappings}. 
We denote the one-sided limits by
$$ 
f(x+) = \lim_{ t \to x \, , t>x } f(t), \quad f(x-) = \lim_{ t \to x\, , t<x } f(t) .
$$  
From now on we write $\displaystyle \lim_{ t \to x+ },\; \lim_{ t \to x- }$ 
instead of $\displaystyle \lim_{ t \to x \, , t>x },\, \lim_{ t \to x \, , t<x }$ respectively.
\smallskip

In~\cite{BDFM:2021}, we introduced the notion of the left and right local quasi-moduli.
For a function $f : [a,b] \to X$ of bounded variation, the left local quasi-modulus at point $x^*$ is
$$
	\NewLeftLocalModul{f}{x^*}{\delta} = \sup{ \big \{ \rho(f(x^*-),f(x)) \ : \ x \in [x^*-\delta,x^*) \cap [a,b] \big\} }, \quad \delta >0\ , \; x^* \in (a,b].
$$
Similarly, the right local quasi-modulus is
$$
	\NewRightLocalModul{f}{x^*}{\delta} = \sup{ \{ \rho(f(x^*+),f(x)) \ : \ x \in (x^*,x^* + \delta] \cap [a,b] \} }, \quad \delta >0\ , \; x^* \in [a,b).
$$
Clearly, for $f \in \BV$ the local quasi-moduli satisfy
$$
	\lim_{ \delta \to 0^+ } \NewLeftLocalModul{f}{x^*}{\delta}=0, \quad x^* \in (a,b],
	\quad \text{and} \quad
	\lim_{ \delta \to 0^+ } \NewRightLocalModul{f}{x^*}{\delta} =0 , \quad x^* \in [a,b) .
$$

\begin{result}\label{Lemma_QuasiMod_f<=QuasiMod_v_f}~\cite[Lemma~2.5]{BDFM:2022}
	Let $f : [a,b] \to X$, $f \in \BV$, then for any $x^*\in (a,b]$ or $[a,b)$, respectively,  and $ \delta>0$ we have
	$$
	\NewLeftLocalModul{f}{x^*}{\delta} \le \NewLeftLocalModul{v_f}{x^*}{\delta}, \quad  \NewRightLocalModul{f}{x^*}{\delta} \le \NewRightLocalModul{v_f}{x^*}{\delta}.
	$$
\end{result}

\subsection {Metric selections of set-valued functions }\label{Sect_MS_MetIntegral}

We consider set-valued functions (SVFs, multifunctions) mapping a compact interval $[a,b] \subset \R$ to $\Comp$. 
The graph of a multifunction $F$ is the set of points in $\R^{d+1}$ 
$$
\Graph(F)= \left \{(x,y) \ : \ y\in F(x),\; x \in [a,b] \right \}.
$$
It is easy to see that if $F \in  \BV$ then $\Graph(F)$ is a bounded set and 
$$
\|F\|_\infty =  \sup_{x\in[a,b]} | F(x) | <  \infty . 
$$
We denote the class of SVFs of bounded variation with compact graphs by~$\calF$.

For a set-valued function $F : [a,b] \to \Comp$, a single-valued function ${s:[a,b] \to \Rd}$ such that $s(x) \in F(x)$ for all $x \in [a,b]$ is called a selection of~$F$.

The notions of chain functions and metric selections are among central notions in our work.
Given a multifunction $F: [a,b] \to \Comp$, a partition $\chi=\{x_0,\ldots,x_n\} \subset [a,b]$, $a=x_0 < \cdots < x_n=b$, and a corresponding metric chain $\phi=(y_0,\ldots,y_n) \in \CH \left ( F(x_0),\ldots,F(x_n) \right )$ (see Definition~\ref{Def_MetChain_MetCombination}), the \textbf{chain function} based on $\chi$ and $\phi$ is
\begin{equation}\label{def_ChainFunc}
	c_{\chi, \phi}(x)= \left \{ \begin{array}{ll}
		y_i, & x \in [x_i,x_{i+1}), \quad i=0,\ldots,n-1,\, \\
		y_n, & x=x_n.
	\end{array}
	\right.
\end{equation}

A selection $s$ of $F$ is called a \textbf{metric selection}, if there is a sequence of chain functions $\{ c_{\chi_k, \phi_k} \}_{k \in \N}$ of~$F$ with ${\lim_{k \to \infty} |\chi_k| =0}$ such that
$$
s(x)=\lim_{k\to \infty} c_{\chi_k, \phi_k}(x) \quad \mbox{pointwisely on} \ [a,b].
$$
We denote the set of all metric selections of $F$ by $\setMS$.

Note that the definitions of chain functions and metric selections imply that a metric selection $s$ of a multifunction $F$ is constant in any open interval where the graph of $s$ stays in the interior of  $\Graph(F)$.

Below we quote several relevant results from~\cite{DFM:MetricIntegral} and~\cite{BDFM:2021}.

\begin{result}\label{Result_MetrSel_InheritVariation}\cite[Theorem~3.6]{DFM:MetricIntegral}
	
	Let $s$ be a metric selection of $F \in \calF$. Then $V_a^b(s) \le V_a^b(F)$ and $\|s\|_\infty\le \|F\|_\infty$.
\end{result}
\begin{result}\label{Result_MetSel_ThroughAnyPoint_Repres}\cite[Corollary~3.7]{DFM:MetricIntegral}
	Let $F \in \calF$. Through any point $\alpha \in \Graph(F)$ there exists a metric selection which we denote by~${\,s_\alpha}$.
	Moreover,  $F$ has a representation by metric selections, namely
	$$
	F(x) = \{ s_\alpha(x) \ :\ \alpha \in \Graph(F)\}, \quad x\in [a,b].
	$$
\end{result}

The last result implies
\begin{corol}\label{Corol_Repres_F}
For $F \in \calF$,
$$
F(x) = \{ s(x) \ :\ s \in  \setMS \}, \quad x\in [a,b].
$$
\end{corol}

\begin{result}\label{Result_LocModuli_s<=LocModuli_v_F}\cite[Theorem~4.9]{BDFM:2021}
	
	Let $F \in \calF$, $s$ be a metric selection of $F$  and $x^*\in[a,b]$. Then
	$$
	\LocalModulCont{s}{x^*}{\delta} \le \LocalModulCont{v_F}{x^*}{ 2\delta}, \quad \delta >0.
	$$
	In particular, if $F$ is continuous at $x^*$, then $s$ is continuous at $x^*$.
\end{result}


\section{The one-sided limits of SVFs and their  metric selections}\label{Sect_Erratum}

In this section we study relationships between the value $F(x)$ of $F \in \calF$ and its one-sided limits at $x \in (a,b)$. We recall and refine some relevant results from Section~7 of our recent paper \cite{BDFM:2021}.

\begin{result}\label{Result_Right-Left_Limits_in_F} \cite[Proposition~7.1]{BDFM:2021}\\
		Let $F\in \calF$ and $x \in (a,b)$, then 
		$$F(x-) \cup F(x+) \subseteq F(x).$$
\end{result}
	
The next result is stated in~\cite{BDFM:2021}, however, the proof in~\cite{BDFM:2021} is incomplete. Here we provide a constructive method of proof that is applicable, with obvious minor changes, to each of the two claims.
	
	\begin{propos}\label{Propos_RightContSelections} \cite[Proposition~7.2]{BDFM:2021}\; \\
		For $F \in \mathcal{F}[a,b]$ 
		\begin{enumerate}\label{Claim_1-}
			\item[(i)]\quad  $F(x-) = \{ s(x-) \ : \ s \in \mathcal{S}(F) \}, \quad x \in (a,b] $,  
			\item[(ii)]\quad  $F(x+) = \{ s(x+) \ : \ s \in \mathcal{S}(F) \}, \quad x \in [a,b)$.
		\end{enumerate}	
	\end{propos}
	\begin{proof}
		(i) Fix $x^*\in(a,b]$. The inclusion ${\{ s(x^*-): \ s\in \mathcal{S}(F)\} \subseteq F(x^*-)}$ follows from the fact that $F(x^*-)$ is the Kuratowski upper limit 
		$\displaystyle \limsup_{x \to x^*-}{F(x)} =\limsup_{x \to x^*-}{\{ s(x): \ s\in \mathcal{S}(F)\}}$. 
		
		Now we prove the reverse inclusion, i.e. that  $F(x^*-) \subseteq \{s(x^*-) \ :\ s\in \mathcal{S}(F)\} $.
		Let $y^*_{-} \in F(x^*-)$, in view of Result~\ref{Result_Right-Left_Limits_in_F}\ $y^*_{-} \in F(x^*)$.  We have to show that there exists $ s \in \setMS$ such that $y^*_{-}=s(x^*-)$. 
		We construct a sequence of chain functions converging to the required $s$ pointwisely.  Let $\{\chi_n\}_{n\in \N}$ be a sequence of partitions of $[a,b]$ such that
		$$ 
		x^* \in \chi_n,\, \forall n \in\N \quad \text{and} \quad  \lim_{n \to \infty} |\chi_{n}| =0,
		$$ 
		and let $\xi_n^{-}$, $\xi_n^{+}$ be the closest points to $x^*$ in the partition $\chi_n$ from the left and from the right respectively.		
		Let $\{c_n\}_{n \in \N}$ be a sequence of chain functions of~$F$ based on $\{\chi_n\}_{n\in \N}$, 
		satisfying for all $n \in \N$: 
		$$
		c_n(x^*)=y^*_{-} ,  \quad c_n(\xi_n^{-}) \in  \PrjXonY{y^*_{-}}{F(\xi_n^{-})} ,  
		\quad c_n(\xi_n^{+}) \in  \PrjXonY{y^*_{-}}{F(\xi_n^{+})}.
		$$
		For a construction of chain functions with these properties see Definition~\ref{Def_MetChain_MetCombination} and~\eqref{def_ChainFunc}.
		By Helly's Selection Principle there exists a subsequence of $\{c_n\}_{n \in \N}$ converging to a metric selection $s$ pointwisely for each $x\in[a, b]$. 
		By construction we have ${\displaystyle \lim_{n \to \infty} c_n(x^*) = y^*_{-} = s(x^*) }$.
		
		Define a multifunction ${\widetilde F^{-} \in \calF}$~by 
		$$
		\widetilde F^{-} (t) =
		\left \{ \begin{array}{ll}
			F(t), & t \neq x^*, \\
			F(x^*-), & t=x^*.
		\end{array}
		\right.
		$$
		Clearly,  $\widetilde F^{-}$ is left continuous at~$x^*$ and therefore $v_{\widetilde F^{-}}$\  is left continuous at $x^*$ as well (see Result~\ref{Result_f_OneSidedCont->v_f_OneSidedCont}). By construction $s \in  \setMSof{F}$ and also $s \in  \setMSof{\widetilde F^{-}}$, since $y^*_{-} \in F(x^*)$ and $y^*_{-} \in 	\widetilde F^{-}(x^*)$, and the above chain functions of $F$ are also chain functions of $\widetilde F^{-}$.
			
		To prove that $y^*_{-}=s(x^*-)$, we estimate $|s(x^*-\delta)-y^*_{-}|$, for $0< \delta < x^*-a$,
		\begin{align*}
			|s(x^*-\delta)-y^*_{-}| &= |s(x^*-\delta)-s(x^*)| \le \Var{x^*-\delta}{x^*}{s} \le \Var{x^*-\delta}{x^*}{\widetilde F^{-}} \le \LeftLocalModul{v_{\widetilde F^{-}}}{x^*}{\delta},
		\end{align*}
		where the second inequality follows from Result~\ref{Result_MetrSel_InheritVariation}. 
		Since
		$\LeftLocalModul{v_{\widetilde F^{-}}}{x^*}{\delta}$ tends to zero as $\delta \to 0^+$,
		we get 
		$$
		s(x^*-) = \lim_{\delta \to 0^+ } s(x^*-\delta) = y^*_{-}=s(x^*) \in F(x^*-).
		$$
		In particular $s$ is left continuous at $x^*$.
		
		The proof of~{(ii)} is similar with obvious minor changes such as replacing $y^*_{-} \in F(x^*-)$  by ${y^*_{+} \in F(x^*+)}$ and $\widetilde F^{-}$ by $\widetilde F^{+}$. In this case the constructed metric selection is right continuous at $x^*$.
	\end{proof}
	
A direct conclusion from Proposition~\ref{Propos_RightContSelections} is 
\begin{corol}\label{corol_contin}
	In the notation and assumptions of Proposition~\ref{Propos_RightContSelections} 
\begin{enumerate}
	\item[(i)]\ If $y^* \in F(x^*-)\cap F(x^*+)$, then there exists a metric selection $s$ satisfying $s(x^*)=y^*$ which is continuous at $x^*$.
	\item[(ii)]\ If $y^* \in F(x^*-)\setminus F(x^*+)$, then there exists a metric selection $s$ satisfying $s(x^*)=y^*$ which is left continuous at $x^*$.
	\item[(iii)]\ If $y^* \in F(x^*+)\setminus F(x^*-)$, then there exists a metric selection $s$ satisfying $s(x^*)=y^*$ which is right continuous at $x^*$.		
\end{enumerate}
\end{corol}

\begin{remark}\label{Remark_cont_1}
	While for $F$ which is left continuous at $x^*$ {\bf all} its metric selections are left continuous at $x^*$ as well (see Theorem~4.7 in~\cite{BDFM:2021}), in the proof of~(ii) in Proposition~\ref{Propos_RightContSelections} we obtain that for $F$ which is right continuous at $x^*$ {\bf there exists} a metric selection which is right continuous at~$x^*$.
\end{remark}

\begin{lemma}\label{lemma_contin}
 Let $F \in \mathcal{F}[a,b]$. If $(x^*, y^*)$ is an interior point of $\Graph(F)$, then {\bf any} metric selection satisfying  $s(x^*)=y^*$ is continuous at $x^*$ and is constant in a small neighborhood of $x^*$.
\end{lemma}

\begin{proof}
Since  $(x^*, y^*)$ is an interior point of $\Graph(F)$ there exists a small open neighborhood of this point, ${I_{x^*}\times I_{y^*}}$, in the interior  of $\Graph(F)$. Thus $y^* \in I_{y^*} \subset F(x)$ for all $x \in I_{x^*}$ and
therefore for any ${ y_1, y_2 \in I_{y^*} }$, $(y_1, y_2) \in \Pair{F(x_1)}{F(x_2)}$ for some $x_1, x_2 \in I_{x^*}$ if and only if $y_1=y_2$. Hence any chain function, with graph passing through $I_{x^*}\times I_{y^*}$, is constant there. In particular, any metric selection $s$ satisfying $s(x^*)=y^*$, being the pointwise limit of chain functions, is also constant in $I_{x^*}$.
\end{proof}

\begin{remark}\label{Remark_Cont_2}
If $y^* \in F(x^*-)\cap F(x^*+)$, but $(x^*, y^*)$ is not an interior point of $\Graph(F)$, then a metric selection~$s$ satisfying $s(x^*)=y^*$ is not necessarily continuous at $x^*$.
\end{remark}


\section{The limit set $A_{F}(x)$}\label{Sect_LimitSet}

In our recent research, the following set surfaced, for a function $F \in \calF$,
$$
A_{F}(x) =\left \{ \frac{1}{2} \left( s(x-) + s(x+) \right) \ : \ s \in \setMS \right \} , \; x \in (a,b).
$$
This set appeared as a limit set of sequences of metric Fourier approximations~\cite{BDFM:2021}, or of metric intergal approximation operators~\cite{BDFM:2022}. 

The set $A_{F}(x)$ extends the well known limit, $\frac{1}{2} \left( f(x-) + f(x+) \right)$, of many integral approximation operators and of Fourier approximations from real-valued functions to set-valued ones.

In~\cite{BDFM:2021} we conjectured that under certain assumptions on $F$ one has
\begin{equation} \label{main_equality}
A_F(x) = \frac{1}{2} F(x-)\oplus \frac{1}{2} F(x+)
\end{equation}
with $\oplus$ as in Definition~\ref{Def_MetChain_MetCombination}.
Notice that~\eqref{main_equality} does not hold for all $F \in \calF$. See the examples below, taken from~\cite[Section~7]{BDFM:2021}.

	\begin{remark}\label{Remark_LimitSet=F}
		For $x\in(a,b)$ a point of continuity of $F$, we have
		$$ 
		A_F(x) = F(x)= \frac{1}{2} F(x-)\oplus \frac{1}{2} F(x+).
		$$
		Indeed, by Result~\ref{Result_LocModuli_s<=LocModuli_v_F} all metric selections of $F$ are continuous at $x$, 
		therefore 
		$$
		A_{F}(x) = \left \{ s(x) : \ s \in \setMS \right \} , \; x \in (a,b),
		$$
		and then by Corollary~\ref{Corol_Repres_F} $A_{F}(x) = F(x)$. 
	
		Moreover, since  $F(x-)=F(x+)=F(x)$ we get  $\frac{1}{2} F(x-)\oplus \frac{1}{2} F(x+)=F(x)$.
	\end{remark}

Now we formulate two properties of $F\in \calF$ that are sufficient to guarantee~\eqref{main_equality} also at discontinuity points.

\noindent {\bf Property 1 }(Minimality of $F$)\label{prop1}

We say that $F\in \calF$ has Property 1 at $\xi \in (a,b)$ if 
$$
	F(\xi) = F(\xi-)\cup F(\xi+).
$$

\begin{remark}\label{Remark_Prop1} ${}$\\
(i) The name "Minimality of $F$" reflects Result~\ref{Result_Right-Left_Limits_in_F}.\\
(ii)  Property~1 holds at continuity points of $F\in \calF$.
\end{remark}

\noindent {\bf Property 2 }

We say that $F \in \calF$ satisfies Property~2 at $\xi \in (a,b)$, if for each pair $(y^-, y^+) \in \Pair{F(\xi-)}{F(\xi+)}$ there exist four sequences 
$\{\xi_n^-\}_{n \in \N}$, $\{\xi_n^+\}_{n \in \N}$, $\{y_n^-\}_{n \in \N} $, $\{y_n^+\}_{n \in \N}$, such that 
$$
(y_n^-, y_n^+) \in \Pair{F(\xi_n^-)}{F(\xi_n^+)} , \; n \in \N,
$$
where $\xi_n^- < \xi < \xi_n^+$,
$\lim_{n \to \infty} \xi_n^- = \xi = \lim_{n \to \infty} \xi_n^+$,
$\lim_{n \to \infty} y_n^- = y^-$,  $\lim_{n \to \infty} y_n^+ = y^+$.

\begin{remark}\label{Remark_Prop2} Property 2 can be written equivalently as
$$
\Pair{F(\xi-)}{F(\xi+)} \subseteq \limsup_{(x,z) \to (\xi-, \xi+)} \Pair{F(x)}{F(z)} .
$$
\noindent (For the definition of $ \limsup$ see Section~\ref{Sect_Prelim_OnSets}.)\\
Recall that for  $F \in \calF$ the inverse inclusion
$$
\limsup_{(x,z) \to (\xi-, \xi+)} \Pair{F(x)}{F(z)} \subseteq \Pair{F(\xi-)}{F(\xi+)}
$$
follows from Corollary~\ref{Corol_Inclusion1}. Thus Property~2 implies the equality
$$
	\limsup_{(x,z) \to (\xi-, \xi+)} \Pair{F(x)}{F(z)} = \Pair{F(\xi-)}{F(\xi+)}.
$$
\end{remark}

\begin{lemma}\label{Lemma_ContImpliesProp2}
 If $\xi \in (a,b)$ is a point of continuity of $F \in \calF$, then $F$ has Property~2 at $\xi$.
\end{lemma}

\begin{proof}
	Since $	F(\xi) = F(\xi-)=F(\xi+)$, any metric pair in $\Pair{F(\xi-)}{F(\xi+)}$ is of the form $(y,y)$, 
	with $y\in F(\xi)$. For a fixed $(y,y) \in \Pair{F(\xi-)}{F(\xi+)}$, take arbitrary sequences $\{\xi_n^-\}_{n \in \N}$, $\{\xi_n^+\}_{n \in \N}$ satisfying $\xi_n^- < \xi < \xi_n^+$,\, 
	$\lim_{n \to \infty} \xi_n^- = \xi = \lim_{n \to \infty} \xi_n^+$. 
	Let $y_n^- \in \PrjXonY{y}{F(\xi_n^-)}$, $y_n^+ \in \PrjXonY{y_n^-}{F(\xi_n^+)}$, $n\in \N$. 
	By construction $ (y_n^-, y_n^+) \in \Pair{F(\xi_n^-)}{F(\xi_n^+)} $.
Since $ (y_n^-, y) \in \Pair{F(\xi_n^-)}{F(\xi)}$, by~\eqref{haus_MetrPair} and the continuity of $F$ at $\xi$ we obtain
$|y_n^- - y| \le \haus( F(\xi_n^-) , F(\xi) ) \to 0$ as $n\to \infty$.

Using the triangle inequality and~\eqref{haus_MetrPair}, we have
$$
|y_n^+ - y| \le |y_n^+ - y_n^- | + |y_n^- - y| \le \haus( F(\xi_n^+) , F(\xi_n^-)) + \haus( F(\xi_n^-) , F(\xi)).
$$
By the triangle inequality for $\haus( F(\xi_n^+) , F(\xi_n^-))$ and the continuity of $F$ at $\xi$ we obtain
$$
|y_n^+ - y| \le \haus( F(\xi_n^+) , F(\xi)) +  2 \, \haus(F(\xi),  F(\xi_n^-)) \to 0 \quad \text{as} \quad n\to \infty,
$$
implying that $\lim_{n \to \infty} y_n^+ = y$.

\end{proof}

The next two examples, taken from~\cite[Section~7]{BDFM:2021}, show that Property~1 and Property~2 are not 
necessarily satisfied by all functions in~$\calF$.

The first example (\cite[Example~7.3]{BDFM:2021})} presents a multifunction $F: [a,b] \to {\mathrm{K}(\R^2)}$, $F \in \mathcal{F}[a,b]$ that does not satisfy Property~1 at a given $\xi \in (a,b)$. Consider
$$
F(t) = \begin{cases}
	B(-2,2), & t \in [a,\xi), \\
	B(-2,2) \cup \{(0,0)\}  \cup B(2,2), & t = \xi, \\
	B(2,2), & t \in (\xi,b],
\end{cases}
$$
where $B(x_1,x_2)$ denotes the closed unit disc with center at $(x_1,x_2)$.
For its metric selection
$$
s(t) = \begin{cases}
	(-2 + \frac{\sqrt{2}}{2}, 2 - \frac{\sqrt{2}}{2}), & t \in (a,\xi),  \\
	(0,0), & t = \xi,\\
	(2 - \frac{\sqrt{2}}{2}, 2 - \frac{\sqrt{2}}{2}), & t \in (\xi,b].
\end{cases}
$$
it is shown that $(s(\xi-), s(\xi+)) \not\in  \Pair{F(\xi-)}{F(\xi+)}$, and that ${\frac{1}{2} ( s(\xi-) + s(\xi+)) = (0, 2 - \frac{\sqrt{2}}{2}) \in A_F(\xi)}$, but does not belong to ${\frac{1}{2} F(\xi-) \oplus \frac{1}{2} F(\xi+)}$. Thus $A_F(\xi) \not\subseteq \frac{1}{2} F(\xi-)\oplus \frac{1}{2} F(\xi+)$. Clearly, $F$ satisfies Property~2 at $\xi$ since $F$ is constant on each of the intervals $[a,\xi)$ and $(\xi,b]$: for any pair $(y^-, y^+) \in \Pair{F(\xi-)}{F(\xi+)}$ and an arbitrary suitable choice of $\{\xi_n^-\}_{n \in \N}$, $\{\xi_n^+\}_{n \in \N}$  we can take constant sequences $y_n^- = y^-$, $y_n^+ = y^+$, $n \in \N$.

It is easy to modify $F$ to a multifunction $\widetilde F$ satisfying both Property~1 and Property~2 at $\xi$, 
$$
\widetilde F(t) = \begin{cases}
	B(-2,2), & t \in [a,\xi), \\
	B(-2,2) \cup B(2,2), & t = \xi, \\
	B(2,2), & t \in (\xi,b].
\end{cases}
$$
\medskip

An example of a multifunction $G: [a,b] \to {\mathrm{K}(\R)}$, $G \in \mathcal{F}[a,b]$ that satisfies Property~1 but does not satisfy Property~2 at a given $\xi \in (a,b)$ can be found in~\cite[Example~7.4]{BDFM:2021}:
$$
G(t) = \begin{cases}
	\left\{ -\frac{1}{4}, 0, \frac{1}{4} \right\}, & t \in [a,\xi),\\
	\left\{ -1, -\frac{1}{4}, 0, \frac{1}{4}, 1 \right\}, & t = \xi, \\
	\left\{ -1 + t - \xi, 1 + t - \xi \right\}, & t \in (\xi,b].
\end{cases}
$$
It is easy to see that $(0,1) \in \Pair{G(\xi-)}{G(\xi+)}$, so that $\frac{1}{2} \in \frac{1}{2} G(\xi-) \oplus \frac{1}{2} G(\xi+)$. We showed that there is no metric selection $s$ of $G$ such that $\frac{1}{2} = \frac{1}{2}(s(x-) + s(x+))$. Thus,  for this function  $A_G(\xi) \not\supseteq \frac{1}{2} G(\xi-)\oplus \frac{1}{2} G(\xi+)$. A similar argument as in~\cite{BDFM:2021} can be used to show that there are no sequences $\{\xi_n^-\}_{n \in \N}$, $\{\xi_n^+\}_{n \in \N}$ with  $\xi_n^- < \xi < \xi_n^+$, $\xi_n^\pm \to \xi$,  
$\{y_n^-\}_{n \in \N} $ with $y_n^- \in G(\xi_n^-)$, $y_n^- \to 0$,  $\{y_n^+\}_{n \in \N}$ with $y_n^+ \in G(\xi_n^+)$, $y_n^+ \to 1$,  such that $(y_n^-, y_n^+) \in \Pair{G(\xi_n^-)}{G(\xi_n^+)}$, $n \in \N$.

A modified function
$$
\widetilde G(t) = \begin{cases}
	\left\{ -\frac{1}{4}, 0, \frac{1}{4} \right\}, & t \in [a,\xi),\\
	\left\{ -1, -\frac{1}{4}, 0, \frac{1}{4}, 1 \right\}, & t = \xi, \\
	\left\{ -1, 1 \right\}, & t \in (\xi,b],
\end{cases}
$$
satisfies both Properties~1 and~2 at $\xi$.

Now we state the main result of the paper. 

\begin{theorem}\label{Theorem_MainResult}
		Let $F \in \calF$ satisfy Property~1 and Property~2 at $x\in (a,b)$. Then
		$$ 
		A_F(x) = \frac{1}{2} F(x-)\oplus \frac{1}{2} F(x+). 
		$$
\end{theorem}

By Remark~\ref{Remark_LimitSet=F} the statement is trivial if $x$ is a point of continuity of $F$. If $x$ is a point of discontinuity, it is a direct consequence of the following two propositions. 

\begin{propos}\label{Prop_FirstInclusion} 
		Let $F \in \calF$ satisfy Property~2 at a point of discontinuity $\xi \in (a,b)$. Then 
		$$ 
		\frac{1}{2} F(\xi-)\oplus \frac{1}{2} F(\xi+) \subseteq A_F(\xi).
		$$
\end{propos}
	
\begin{propos}\label{Propos_SecondInclusion} 
		Let $F \in \calF$ satisfy Property~1 and Property~2 at a point of discontinuity $\xi \in (a,b)$. Then 
		$$ A_F(\xi)  \subseteq \frac{1}{2} F(\xi-)\oplus \frac{1}{2} F(\xi+).$$
\end{propos}
The proofs of these propositions are postponed to the Appendix.

\medskip
 

\medskip



\medskip


\appendix
{\Large\bf\title{Appendix}}
\smallskip

Here we present the proofs of the two propositions of~Section 4, based on a technical lemma stated and proved next.

\section{Lemma}

\begin{lemma}\label{Lemma_Var(chain)} 
	Let $F \in \calF$, and let $c$ be a chain function of $F$  based on a partition ${\chi = \{a = x_0 < \cdots < x_n = b\} }$.  
	\begin{enumerate}
		\item[{(i)}]  If $a \le \alpha < x_j \le \beta \le b$ with some $x_j \in \chi$, then $V_{x_j}^{\beta}(c) \le \NewRightLocalModul{v_F}{\alpha}{\beta -\alpha}$,
		\item[{(ii)}]  If $a \le \alpha \le x_j < \beta \le b$ with some $x_j \in \chi$, then $V_{\alpha}^{x_j}(c) \le \NewLeftLocalModul{v_F}{\beta}{\beta -\alpha + |\chi|}$.
	\end{enumerate}
\end{lemma}

\begin{proof}
	${}$\\
	(i). Let $\beta \in [x_i, x_{i+1})$ with some $i \ge j$. If $i=j$, then $c(t) = c(x_j)$ for all $t \in [x_j,\beta]$, and thus $V_{x_j}^{\beta}(c) = 0$. Therefore suppose that $i > j$. Then $c(t) = c(x_i)$ for all $t \in [x_i,\beta]$, and we have
	\begin{align*}
		V_{x_j}^{\beta}(c) & = V_{x_j}^{x_i}(c) 
		=  \sum_{k=j}^{i-1} |c(x_{k+1}) - c(x_k)|	
		\le \sum_{k=j}^{i-1}  \haus(F(x_{k+1}), F(x_k)) \\&
		\le V_{x_j}^{x_i} (F) \le  V_{x_j}^{\beta}(F)
		\le v_F(\beta) - v_F(\alpha+)  = \NewRightLocalModul{v_F}{\alpha}{\beta - \alpha}.
	\end{align*}
	(ii).
	If $\alpha = x_j$, then $V_{\alpha}^{x_j}(c) =0$. So assume that $\alpha < x_j$, then $\alpha \in [x_{i-1}, x_i)$ with some $i \le j$. We have $c(t) = c(x_{i-1})$ for all $t \in [x_{i-1},\alpha]$, and therefore
	\begin{align*}
		V_{\alpha}^{x_j}(c) & 
		= V_{x_{i-1}}^{x_{j}}(c) 
		= \sum_{k = i-1}^{j-1} |c(x_{k+1}) - c(x_k)|
		\le  \sum_{k = i-1}^{j-1} \haus(F(x_{k+1}), F(x_k))\\&
		\le   V_{x_{i-1}}^{x_{j}}(F)
		\le v_F(\beta-) - v_F(x_{i-1})  
		= \NewLeftLocalModul{v_F}{\beta}{\beta - x_{i-1}} 
		\le \NewLeftLocalModul{v_F}{\beta}{\beta - \alpha+|\chi|}.
	\end{align*}
\end{proof}
\medskip


\section{Proof of Proposition~\ref{Prop_FirstInclusion}}
First we recall the statement of the proposition and then prove it.
\medskip

\textit{Let $F \in \calF$ satisfy Property~2 at a point of discontinuity $\xi \in (a,b)$. Then 
	$$ 
	\frac{1}{2} F(\xi-)\oplus \frac{1}{2} F(\xi+) \subseteq A_F(\xi).
	$$
}
\begin{proof}
	It suffices to show that for any $(y^-, y^+) \in \Pair{F(\xi-)}{F(\xi+)}$ there exists a metric selection $s \in \setMS$ satisfying
	$s(\xi-) = y^-$ and $s(\xi+) = y^+$.
	
	Let $\{\xi_n^-\}_{n \in \N} $, $\{\xi_n^+\}_{n \in \N}$ ,
	$\{y_n^-\}_{n \in \N} $, $\{y_n^+\}_{n \in \N}$ be sequences guaranteed by Property 2.
	Consider a sequence of partitions  $\{ \chi_n \}_{n\in \N}$ such that $\xi \notin \chi_n$, $n \in \N$, 
	and $\xi_n^- , \xi_n^+  \in \chi_n$ are the closest points to $\xi$ in $\chi_n$ from the left and from the right respectively. The remaining points in $\chi_n$ are determined arbitrarily in a way that $\displaystyle \lim_{n \to \infty } | \chi_n | =0$. We now define a chain function $c_n$ based on $\chi_n$. Let 
	$c_n(\xi_n^-)=y_n^-$ and  $c_n(\xi_n^+)=y_n^+$. Then $(y_n^-, y_n^+) \in \Pair{F(\xi_n^-)}{F(\xi_n^+)}$, $\displaystyle \lim_{n \to \infty } y_n^- = y^-$  and 
	$\displaystyle \lim_{n \to \infty } y_n^+ = y^+$ by Property~2. 
	The remaining values of $c_n$ on $\chi_n$ are determined by taking a projection of $c_n(x_i)$ on $F(x_{i-1})$ to the left of $\xi_n^{-}$ and a projection of $c_n(x_i)$ on $F(x_{i+1})$ to the right of $\xi_n^{+}$, namely for $x_i \in \chi_n$
	$$
	c_n(x_{i-1}) \in \PrjXonY{c_n(x_i)}{F(x_{i-1})}, \; \ x_i \le \xi_n^- , \quad 
	c_n(x_{i+1}) \in \PrjXonY{c_n(x_i)}{F(x_{i+1})}, \; \ x_i \ge \xi_n^+.
	$$
	By Helly's Selection Principle, and passing to a subsequence if necessary, we obtain a metric selection $s$ of $F$, as the pointwise  limit of that subsequence.
	
	First we prove that $s(\xi+) = y^+$.
	Using the triangle inequality and the fact that $c_n(\xi_n^+)=y_n^+$, we get for a given fixed small $\delta>0$ and $n$ so large that
	$\xi_n^+<\xi+\delta$
	\begin{equation}\label{Formula_1}
		\begin{split}
			|s(\xi+\delta)-y^{+}| & \le
			|s(\xi+\delta)-c_n(\xi+\delta)|+|c_n(\xi+\delta)-c_n(\xi_n^+)|+|c_n(\xi_n^+) - y^{+}| \\&
			\le |s(\xi+\delta)-c_n(\xi+\delta)|+ \Var{\xi_n^+}{\xi+\delta}{c_n} +|y_n^{+}-y^{+}|.
		\end{split}
	\end{equation}	
	By Lemma~\ref{Lemma_Var(chain)}~(i) we get ${\Var{\xi_n^+}{\xi+\delta}{c_n} \le \NewRightLocalModul{v_F}{\xi}{\delta}}$.
	We further increase $n$, if necessary, to guarantee that $|s(\xi+\delta)-c_n(\xi+\delta)| < \delta$, ${|y_n^{+}-y^{+}| < \delta}$. This is possible since $\lim_{n \to \infty}c_n(\xi+\delta)=s(\xi+\delta)$ and $\lim_{n \to \infty}y_n^{+} =y^{+}$.  Therefore~\eqref{Formula_1} leads to
	$$ 
	|s(\xi+\delta)-y^{+}| < 2\delta + \NewRightLocalModul{v_F}{\xi}{\delta}.
	$$
	The right-hand side in the above formula tends to zero as $\delta \to 0^+$, thus $s(\xi+)=y^+$. 
	
	Now we show that $s(\xi-) = y^-$. For that we use an analogue of equation~\eqref{Formula_1} for a fixed small $\delta > 0$ and $n$ so large that ${\xi-\delta < \xi_n^-}$
	\begin{equation}\label{Formula_2}
		\begin{split}
			|s(\xi-\delta)-y^{-}| & \le
			|s(\xi-\delta)-c_n(\xi-\delta)|+|c_n(\xi-\delta)-c_n(\xi_n^-)|+|c_n(\xi_n^-) - y^{-}| \\&
			\le |s(\xi-\delta)-c_n(\xi-\delta)|+ V_{\xi-\delta}^{\xi_n^-}(c_n) +|y_n^{-}-y^{-}|.
		\end{split}
	\end{equation}
	Arguing as above we obtain by Lemma~\ref{Lemma_Var(chain)}~(ii) 
	${	V_{\xi-\delta}^{\xi_n^-}(c_n)\le \NewLeftLocalModul{v_F}{\xi}{\delta + |\chi_n|}}$.
	We further increase $n$, if necessary, to guarantee that $|\chi_n|<\delta$, $|s(\xi-\delta)-c_n(\xi-\delta)| < \delta$, ${|y_n^{-}-y^{-}| < \delta}$. 
	Then~\eqref{Formula_2} leads to
	$$ 
	|s(\xi-\delta)-y^{-}| < 2\delta + \NewLeftLocalModul{v_F}{\xi}{2\delta},
	$$
	whose right-hand side tends to zero as $\delta \to 0^+$. Hence $s(\xi-)=y^-$. 
\end{proof}
\medskip


\section{Proof of Proposition~\ref{Propos_SecondInclusion}}
First we recall the statement of the proposition and then prove it.
\medskip

\textit{
	Let $F \in \calF$ satisfy Property~1 and Property~2 at a point of discontinuity $\xi \in (a,b)$. Then 
	$$ A_F(\xi)  \subseteq \frac{1}{2} F(\xi-)\oplus \frac{1}{2} F(\xi+).$$
}
\begin{proof}
	We have to show that each $s\in\setMS $ satisfies	$( s(\xi-), s(\xi+) ) \in \Pair{F(\xi-)}{F(\xi+)}$. 
	Let $\{ \chi_n \}_{n\in \N}$ be a sequence of partitions and $\{ c_n\}_{n\in \N}$ be a corresponding sequence of chain 
	functions defining $s$, namely $ \displaystyle \lim_{n \to \infty } c_n(x) =s(x)$, $x \in [a,b]$.
	
	\noindent 
	
	Denote by $\xi_n^-$ and $\xi_n^+$ the closest points to $\xi$ in $\chi_n$ from the left and from the right respectively with ${\xi_n^- < \xi < \xi_n^+}$, and denote ${c_n(\xi_n^-)=y_n^-}$, ${c_n(\xi_n^+)=y_n^+}$ and $c_n(\xi) = y_n$. 
	Taking convergent subsequences, if necessary, we obtain that the two limits
	$\displaystyle \lim_{n \to \infty } y_n^-$ and $\displaystyle\lim_{n \to \infty } y_n^+$
	exist; we denote them by $y^-$ and $y^+$ respectively. 
	
	We consider several cases.
	\smallskip
	
	Case 1: There exists a subsequence of $\{\chi_n\}_{n \in \N}$ with $\xi \notin \chi_n$. For simplicity we denote it again by $\{\chi_n\}_{n \in \N}$. 
	In this case ${(y_n^-, y_n^+) \in \Pair{F(\xi_n^-)}{F(\xi_n^+)} }$ by the definition of chain functions.  
	By Lemma~\ref{Lemma_Limit_of_MP_is_MP} $(y^-, y^+) \in \Pair{F(\xi-)}{F(\xi+)}$.
	One can prove that $s(\xi-)=y^-$ and $s(\xi+)=y^+$ following the lines of the proof of Proposition~\ref{Prop_FirstInclusion}.
	\smallskip
	
	Case 2: $\xi \notin \chi_n$ only for finitely many $n \in \N$. In this case there is a subsequence of $\{\chi_n\}_{n \in \N}$ with ${\xi \in \chi_n}$, which we denote again by $\{\chi_n\}_{n \in \N}$. 
	By the definition of chain functions
	${(y_n^-, y_n) \in \Pair{F(\xi_n^-)}{F(\xi)}}$ and ${(y_n, y_n^+) \in \Pair{F(\xi)}{F(\xi_n^+)}}$.
	By construction $y_n \in F(\xi)$ and $\displaystyle \lim_{n \to \infty } y_n = s(\xi)$.
	
	By Property~1 we have two subcases:
	\begin{itemize}
		\item[(2a)] $y_n \in  F(\xi-)$ for infinitely many $n \in \N${\color{magenta}.}
		\item[(2b)] $y_n \in  F(\xi-)$ for only finitely many  $n \in \N$. In this case
		$y_n \in  F(\xi+)$ for infinitely many $n \in \N${\color{magenta}.}
	\end{itemize}
	
	Subcase (2a): We consider a subsequence of   $\{\chi_n\}_{n \in \N}$ with $y_n \in  F(\xi-)$, which we denote again by $\{\chi_n\}_{n \in \N}$.
	We have $(y_n, y_n^+) \in \Pair{F(\xi)}{F(\xi_n^+)}$. By Lemma~\ref{Lemma_Help} ${(y_n, y_n^+) \in \Pair{F(\xi-)}{F(\xi_n^+)} }$ and by Lemma~\ref{Lemma_Limit_of_MP_is_MP}\, $(s(\xi), y^+) \in \Pair{F(\xi-)}{F(\xi+)}$. To obtain the claim, it is enough to show that $s(\xi-)=s(\xi)$,\ $s(\xi+)=y^+$.
	
	To prove that $s(\xi-)=s(\xi)$, fix a small $\delta >0$ and let $n$ be so large that $\xi - \delta < \xi_n^-$. We estimate
	\begin{equation}\label{Formula_3}
		\begin{split}
			|s(\xi-\delta)-s(\xi)| & \le
			|s(\xi-\delta)-c_n(\xi-\delta)|+|c_n(\xi-\delta)-c_n(\xi_n^-)| + |c_n(\xi_n^-)- c_n(\xi)|+ |c_n(\xi)-s(\xi)|\\&
			\le |s(\xi-\delta)-c_n(\xi-\delta)|+ V_{\xi-\delta}^{\xi_n^-}(c_n) +|y_n^- - y_n| + |c_n(\xi)-s(\xi)|.
		\end{split}
	\end{equation}
	By Lemma~\ref{Lemma_Var(chain)}~(ii) we have 
	$V_{\xi-\delta}^{\xi_n^-}(c_n) \le \NewLeftLocalModul{v_F}{\xi}{\delta + |\chi_n|}$. 
	Since $(y_n^-, y_n) \in \Pair{F(\xi_n^-)}{F(\xi)}$,  Lemma~\ref{Lemma_Help} implies that $(y_n^-, y_n) \in \Pair{F(\xi_n^-)}{F(\xi-)}$.
	By~\eqref{haus_MetrPair} and by Result~\ref{Lemma_QuasiMod_f<=QuasiMod_v_f}
	$$
	|y_n^- - y_n| \le \haus (F(\xi_n^-), F(\xi-)) \le \NewLeftLocalModul{F}{\xi}{|\chi_n|} 
	\le \NewLeftLocalModul{v_F}{\xi}{|\chi_n|}.
	$$
	We further increase $n$, if necessary, to guarantee that $|\chi_n|<\delta$, $|s(\xi-\delta)-c_n(\xi-\delta)| < \delta$, $|s(\xi)-c_n(\xi)| < \delta$. 
	Then~\eqref{Formula_3} gives
	$$ 
	|s(\xi-\delta)-s(\xi)| < 2\delta + \NewLeftLocalModul{v_F}{\xi}{2\delta} + \NewLeftLocalModul{v_F}{\xi}{\delta},
	$$
	and the right-hand side tends to zero as $\delta \to 0^+$. Hence $s(\xi-)=s(\xi)$. 
	
	By similar arguments we show that $s(\xi+)=y^+$. Indeed, for a small fixed $\delta >0$ 
	and large $n$ we have 
	\begin{equation}\label{Formula_11} \nonumber
		\begin{split}
			|s(\xi+\delta)-y^{+}| & \le
			|s(\xi+\delta)-c_n(\xi+\delta)|+|c_n(\xi+\delta) - c_n(\xi_n^+)|+|c_n(\xi_n^+) - y^{+}| \\&
			\le |s(\xi+\delta)-c_n(\xi+\delta)|+ \Var{\xi_n^+}{\xi+\delta}{c_n} +|y_n^{+}-y^{+}|.
		\end{split}
	\end{equation}
		Using the estimate $\Var{\xi_n^+}{\xi+\delta}{c_n} \le  \NewRightLocalModul{v_F}{\xi}{\delta}$, obtained by Lemma~\ref{Lemma_Var(chain)} (i), the facts that ${\lim_{n \to \infty} y_n^+ = y^+}$, 
		${\lim_{n \to \infty} c_n(\xi + \delta) = s(\xi + \delta)}$, and arguing as above, we arrive at
		$s(\xi+)=y^+$.
		
	\medskip
	
	Subcase (2b): We consider a subsequence of   $\{\chi_n\}_{n \in \N}$ with $y_n \in  F(\xi+)$, which we denote again by $\{\chi_n\}_{n \in \N}$. 
	We have  $(y_n^-, y_n) \in \Pair{F(\xi_n^-)}{F(\xi)} $. By Lemma~\ref{Lemma_Help} $(y_n^-, y_n) \in \Pair{F(\xi_n^-)}{F(\xi+)}$ and by Lemma~\ref{Lemma_Limit_of_MP_is_MP} $(y^-, s(\xi) ) \in \Pair{F(\xi-)}{F(\xi+)}$.
	Here we have to show that $s(\xi-)=y^-$, $s(\xi+)=s(\xi)$. The proof is similar to the proof of subcase (2a).
	
	To show that $s(\xi-)=y^-$ we take a small fixed $\delta > 0$ and estimate for large $n$ 
	\begin{equation}\label{Formula_22} \nonumber
		\begin{split}
			|s(\xi-\delta)-y^{-}| & \le
			|s(\xi-\delta)-c_n(\xi-\delta)|+|c_n(\xi-\delta) - c_n(\xi_n^-)|+|c_n(\xi_n^-) - y^{-}| \\&
			\le |s(\xi-\delta)-c_n(\xi-\delta)|+ \Var{\xi-\delta}{\xi_n^-}{c_n} +|y_n^{-}-y^{-}|.
		\end{split}
	\end{equation}
		To arrive at $ s(\xi-)=y^-$, we use the estimate $V_{\xi-\delta}^{\xi_n^-}(c_n) \le  \NewLeftLocalModul{v_F}{\xi}{\delta + |\chi_n|}$, obtained by Lemma~\ref{Lemma_Var(chain)} (ii), the facts $\lim_{n \to \infty} c_n(\xi - \delta) = s(\xi - \delta)$, $\lim_{n \to \infty} y_n^- = y^-$ and arguments as above.

	
	 Finally, to show that $s(\xi+)=s(\xi)$,  we take a small fixed $\delta > 0$ and estimate for large $n$
		\begin{equation} \nonumber
			\begin{split}
				|s(\xi + \delta)-s(\xi)| & \le
				|s(\xi + \delta) - c_n(\xi + \delta)|+|c_n(\xi + \delta) - c_n(\xi_n^+)| + |c_n(\xi_n^+) - c_n(\xi)| + |c_n(\xi)-s(\xi)|\\&
				\le |s(\xi + \delta)-c_n(\xi + \delta)|+ V_{\xi_n^+}^{\xi + \delta}(c_n) + |y_n^+ - y_n|  + |c_n(\xi) - s(\xi)|.
			\end{split}
		\end{equation}
		By  Lemma~\ref{Lemma_Var(chain)} (i) $V_{\xi_n^+}^{\xi + \delta}(c_n) \le \NewRightLocalModul{v_F}{\xi}{\delta}$.
		Since $(y_n, y_n^+) \in \Pair{F(\xi)}{F(\xi_n^+)}$, by Lemma~\ref{Lemma_Help} $(y_n, y_n^+) \in \Pair{F(\xi+)}{F(\xi_n^+)}$. By~\eqref{haus_MetrPair} and Result~\ref{Lemma_QuasiMod_f<=QuasiMod_v_f}
		$$
		|y_n^+ - y_n| \le \haus (F(\xi_n^+), F(\xi+)) \le \NewRightLocalModul{F}{\xi}{|\chi_n|} 
		\le \NewRightLocalModul{v_F}{\xi}{|\chi_n|}.
		$$
		Now we use the facts that $\lim_{n \to \infty} c_n(\xi + \delta) = s(\xi + \delta)$, $\lim_{n \to \infty} c_n(\xi) = s(\xi)$ and the arguments as above to arrive at $s(\xi+)=s(\xi)$.
\end{proof}

	

\begin{thebibliography}{99}

\bibitem{AmbrosioTilli:Topics}
L.~Ambrosio, P.~Tilli,
Topics on Analysis in Metric Spaces,
Oxford University Press, Oxford (2004).



\bibitem{Artstein:MA}
Z.~Artstein,
Piecewise linear approximations of set-valued maps,
\textit{J. Approx. Theory} 56 (1989), 41--47.
%

\bibitem{AubinFrankowska:90}
J.-P.~Aubin, H.~Frankowska, Set-valued analysis, 
Birkhäuser Boston Inc., Boston (1990).
%


\bibitem{AUM:65}
R. J. Aumann,
Integrals of set-valued functions, \textit{J. Math. Anal. Appl.} 12 (1965), 1--12.
%

\bibitem{Babenko:2016}
V. F.~Babenko, V. V. Babenko, M. V. Polishchuk, 
Approximation of some classes of set-valued periodic functions by generalized trigonometric polynomials, 
\textit{Ukrainian Math.~J.} 68(4) (2016), 502--514.
%

\bibitem{BaierPerria:11}
R.~Baier,  G.~Perria, 
Set-valued Hermite interpolation,
\textit{J. Approx. Theory} 163 (2011),  1349--1372.
%

\bibitem{BDFM:2019}
E. E.~Berdysheva, N.~Dyn, E.~Farkhi, A.~Mokhov,
Metric approximation of set-valued functions of bounded variation,
\textit{J. Comput. Appl. Math.} 349 (2019), 251--264.
%

\bibitem{BDFM:2021}
E. E.~Berdysheva, N.~Dyn, E.~Farkhi, A.~Mokhov,
Metric Fourier approximation of set-valued functions of bounded variation,
\textit{J. Fourier Anal. Appl. }  (2021) 27:17.
%

\bibitem{BDFM:2022}
E. E.~Berdysheva, N.~Dyn, E.~Farkhi, A.~Mokhov,
Metric approximation of set-valued functions of bounded variation by integral operators,
submitted.
%



\bibitem{Campiti:2019}
M.~Campiti,
Korovkin-type approximation in spaces of vector-valued and set-valued functions,
\textit{Appl. Anal.} 98 (2019), 2486--2496.
%

%

\bibitem{Chistyakov:On_BV-mappings}
V. V.~Chistyakov,
On mappings of bounded variation, 
\textit {J. Dyn. Control Syst.} 3 (1997), 261--289.
%



\bibitem{DynFarkhi:00}
N.~Dyn, E.~Farkhi, 
Spline subdivision schemes for convex compact sets,
\textit{J. Comput. Appl. Math.} 119(1-2) (2000), 133--144.
%

\bibitem{DynFarkhi:01}
N.~Dyn, E.~Farkhi,  
Spline subdivision schemes for compact sets with metric averages, 
In: Trends in approximation theory (Nashville, TN, 2000), Innov. Appl. Math., 
Vanderbilt Univ. Press, Nashville, TN, 2001, 93--102.
%

\bibitem{DF:04}
N.~Dyn, E.~Farkhi,
Set-valued approximations with Minkowski averages --- convergence and convexification rates,
\textit{Numer. Funct. Anal. Optim.} 25 (2004), 363--377.
%

\bibitem{DFM:Chains}
N.~Dyn, E.~Farkhi, A.~Mokhov,
Approximations of set-valued functions by metric linear operators,
\textit{Constr. Approx.} 25 (2007), 193--209.
%

\bibitem{DFM:07serdica}
 N.~Dyn, E.~Farkhi, A.~Mokhov, 
Approximation of univariate set-valued functions—an overview. 
Serdica Math. J. 33 (2007), no. 4, 495–514.

\bibitem{DFM:Book_SV-Approx}
N.~Dyn, E.~Farkhi, A.~Mokhov,
Approximations of Set-Valued Functions: Adaptation of Classical Approximation Operators,
Imperial College Press, London (2014).
%

\bibitem{DFM:MetricIntegral}
N.~Dyn, E.~Farkhi, A.~Mokhov,
The metric integral of set-valued functions, 
\textit{Set-Valued Var. Anal.} 26 (2018), 867--885.
%

\bibitem{DynMokhov:06}
N.~Dyn, A.~Mokhov, 
Approximations of set-valued functions based on the metric average, 
\textit{Rend. Mat. Appl., VII. Ser.} 26(3-4)  (2006), 249--266.
%



\bibitem{Lempio:95}
 F.~Lempio, 
 Set-valued interpolation, differential inclusions, and sensitivity in optimization,
In: Recent developments in well-posed variational problems, Math. Appl., 331, 
Kluwer Acad. Publ., Dordrecht, 1995, 137--169.
%

\bibitem{Muresan:SVApprox2010}
M.~Mure\c{s}an, 
Set-valued approximation of multifunctions, 
\textit{Stud. Univ. Babe\c{s}-Bolyai, Math.}
55 (1) (2010), 107--148.
%

\bibitem{MNikolskii:Opt90}
M. S.~Nikolskiĭ,  
Approximation of continuous convex-valued multivalued mappings (Russian), 
\textit{Optimization} 21(2) (1990), 209--214.
%

%


%


\bibitem{RockWets}
R.~T.~Rockafellar, R.~Wets,
Variational Analysis, Springer-Verlag, Berlin (1998).
%

\bibitem{SendovPopov}
B.~Sendov, V.~Popov, The Averaged Moduli of Smoothness: Applications in Numerical Methods and Approximation, Wiley (1988).
%

\bibitem{S:93}
R.~Schneider,
Convex Bodies: The Brunn--Minkowski Theory, Cambridge University Press, Cambridge (1993).
%


\bibitem{VIT:79}
R.~A.~Vitale,
Approximations of convex set-valued functions,
\textit{J. Approx. Theory}  26 (1979), 301--316.
%



\end{thebibliography}
\end{document}